\UseRawInputEncoding
\documentclass[12pt,psfig]{article}
\usepackage{float}
\usepackage{color}
\usepackage{subfigure}
\usepackage{amsmath}
\usepackage{amsthm}
\usepackage{amsfonts}
\usepackage{amssymb,latexsym}
\usepackage[all]{xy}
\usepackage{graphicx}
\usepackage[mathscr]{eucal}
\usepackage{verbatim}
\usepackage{hyperref}

\theoremstyle{plain}

    \newtheorem{thm}{Theorem}[section]
       
       \newtheorem{lem}{Lemma}[section]
       \newtheorem{defn}{Definition}[section]

\numberwithin{equation}{section}

\usepackage{graphicx}
\usepackage{pict2e}

\begin{document}
\title{Invariant manifolds for  stochastic delayed partial differential equations of parabolic type }
\author{Wenjie Hu$^{1,2}$,  Quanxin Zhu$^{1,3}$\footnote{Corresponding author.  E-mail address: zqx22@126.com (Quanxin Zhu)}, Tom\'{a}s
Caraballo$^{4}$.
\\
\small  1. The MOE-LCSM, School of Mathematics and Statistics,  Hunan Normal University,\\
\small Changsha, Hunan 410081, China\\
\small  2. Journal House, Hunan Normal University, Changsha, Hunan 410081, China\\
\small  3. the Key Laboratory of Control and Optimization of Complex Systems, College of Hunan Province,\\
\small Hunan Normal University, Changsha 410081, China.\\
\small 4 Dpto. Ecuaciones Diferenciales y An\'{a}lisis Num\'{e}rico, Facultad de Matem\'{a}ticas,\\
\small  Universidad de Sevilla, c/ Tarfia s/n, 41012-Sevilla, Spain
}

\date {}
\maketitle

\begin{abstract}
The aim of this paper is to prove  the existence  and  smoothness of stable and unstable invariant manifolds for a stochastic delayed partial differential equation of parabolic type. The stochastic delayed partial differential equation is  firstly transformed  into a random delayed partial differential equation by a conjugation, which is then recast into a Hilbert space. For the auxiliary equation, the variation of constants formula holds and we show the existence of Lipschitz continuous stable and unstable manifolds by the Lyapunov-Perron method. Subsequently, we prove the smoothness of these invariant manifolds under appropriate spectral gap condition by carefully investigating the smoothness of auxiliary equation, after which, we  obtain  the invariant manifolds of the original equation by projection and inverse transformation. Eventually, we illustrate the  obtained theoretical results by their application to a stochastic single-species population model.
\end{abstract}

\bigskip


\section{Introduction}
Invariant manifolds and invariant foliations  are of great significance  in the study of dynamical behavior and geometric properties of nonlinear dynamical systems. For example, the local bifurcation of a high-dimensional system can be studied by reducing it into a central manifold. The intersection of invariant stable manifolds and invariant unstable manifolds can be used to study the existence of homoclinic orbits and heteroclinic orbits, which is a key ingredient in  proving existence of chaotic attractors. The smooth conjugation between different flows can be given by invariant foliations.

Because of the significant roles the invariant manifolds  played in analyzing the dynamics of deterministic or random dynamical systems (RDSs), they have drawn vast attention from mathematicians in the past decades.  As to the pioneer works of invariant manifolds and invariant foliations for finite dimensional deterministic dynamical systems generated by ordinary differential equations or difference equations, the readers are refereed to  \cite{HPS77}. In  \cite{CL88,CLK88}, Chow and Lu extended the general framework for proving the existence and smoothness of  invariant manifolds and central unstable manifolds  for finite dimensional deterministic dynamical systems to semiflows in general Banach spaces generated by partial differential equations. Both papers \cite{CHT97} and  \cite{CLL91}  proved the existence of smooth invariant foliations in infinite dimensional spaces. In \cite{BLZ98} and \cite{BLZ00}, the authors further studied the geometric and analytical properties of invariant manifolds and invariant foliations of semiflows in Banach spaces, including persistence and uniform hyperbolicity.

For RDSs, Arnold \cite{AR98} summarized  early research results of finite dimensional RDSs, which includes basic theory, multiplicative ergodic theorem and smooth dynamical system theory (including invariant manifolds and bifurcation theory). Duan, Lu  and Schmalfuss \cite{DLS03,DLS04} obtained Lipschitzian invariant manifolds and smooth invariant manifolds of stochastic partial differential equations (SPDEs)  with additive or multiplicative noise by the Hadamard graph transformation method and the Lyapunov Perron method respectively. Lu and Schmalfuss \cite{LS07} studied the existence and smoothness of invariant manifolds for stochastic wave equations with a nonlinear multiplicative noise. Mohammed, Zhang and Zhao \cite{MZZ08} established local stable manifold theorems for semilinear stochastic evolution equations and SPDEs with additive or linear infinite dimensional noises by Ruelle's multiplicative ergodic theory in \cite{RD} and the saddle point property for hyperbolic linear random cocycles obtained  in \cite{M90} and  \cite{MS4}. Caraballo et. al. \cite{CDL09} applied the results in \cite{MZZ08} to random partial equations and SPDEs with linear multiplicative noise. For stochastic delayed differential equations (SDDEs) without spatial structure,  \cite{M90} and  \cite{MS4} established the local stable manifold theorems for linear SDDEs, while   \cite{MS2003} \cite{MS2004} obtained corresponding results for nonlinear SDDEs. Very recently, in \cite{GRS} and \cite{GR}, the authors established the multiplicative ergodic theory and local stable manifold theorems for a singular SDDE  with delay in the noise term by using Lyons's rough path theory.  With respect to the abstract RDSs, Li and Lu \cite{LL05} studied the smoothness of stable and unstable manifolds and the existence of stable and unstable invariant foliations for non-uniformly hyperbolic RDSs in finite dimensional space. Lian and Lu \cite{LL10} considered the existence and smoothness of Lyapunov exponents, stable and unstable manifolds of quasi non-uniform hyperbolic discrete RDSs in general Banach spaces. Shen, Lu and Zhang \cite{SLZ20} considered the existence of invariant manifolds, invariant foliations, H\"older smoothness and the stability of discrete RDSs in Hilbert Spaces.

Although the invariant manifolds and their smoothness  for deterministic and stochastic partial differential equations, SDDEs as well as abstract RDSs have been widely studied, the investigation of invariant manifolds for SPDEs with delay is scarce. Indeed, even for deterministic partial functional differential equations (PFDEs) with finite delays, the Lyapunov Perron method needs to be improved due to the lack of a general  variation of  constants formula. In \cite{NW04},  the authors established the general framework for proving existence of invariant manifolds for nonautonomous dynamical systems in Banach spaces by means of graph transformation, which is adopted to show  the existence of invariant manifolds for nonautonomous PFDEs. The authors further established a variation of  constants formula in \cite{HMN} for PFDEs with infinite delays, which was adopted to investigate invariant manifolds for PFDEs with infinite delays in \cite{MV} by the Lyapunov Perron method. However, the natural question how  to establish the Lyapunov Perron method to obtain existence and  smoothness of invariant manifolds for deterministic or stochastic PFDEs with finite delays  remains unsolved, which is the main object of this paper.

We consider the following stochastic retarded  partial differential equation with an additive noise
\begin{equation}\label{1}
 \displaystyle\frac{du(t)}{dt}=A u(t)-\mu u(t)+L u_{t} +f\left(u_{t}\right)+\sum_{j=1}^{m} g_{j}\frac{d w_{j}(t)}{dt}.
\end{equation}
Here, $u(t)\in \mathbb{X}$ and $\mathbb{X}$ is an arbitrary  Hilbert space with norm $\|\cdot\|_{\mathbb{X}}$ and inner product $(\cdot,\cdot)_{\mathbb{X}}$.  $A: D(A)\subset H\rightarrow H$ is a closed densely defined linear operator that generates a strongly continuous semigroup $S(t)$ on $\mathbb{X}$, $u_{t}$ is an element of $\mathcal{L}$ defined by $u_{t}(\xi)=$ $u(t+\xi)$ for $\xi \in [-\tau, 0]$, where $\mathcal{L}\triangleq L^2([-\tau,0], \mathbb{X})$ is the Hilbert space consisting  of all square Lebesgue  integrable functions from $[-\tau, 0]$ to $\mathbb{X}$ equipped with the  norm  $\|\varphi\|_{\mathcal{L}}=[\int_{-\tau}^0\|\varphi(s)\|_{\mathbb{X}}^2ds]^{1/2}$  for all $\varphi\in \mathcal{L}$.  $L: \mathcal{L} \mapsto \mathbb{X}$ is a bounded linear operator and $f : \mathcal{L} \mapsto \mathbb{X}$ is an everywhere defined Lipschitz continuous operator. $\{g_j\}_{j=1}^{m}$ with $g_j\in \mathbb{X}$ stands for the intensity and the shape of noise, $\{w_j\}_{j=1}^{m}$ are mutually independent two-sided real-valued Wiener process on an appropriate probability space to be specified below. We first  transform the delayed partial differential equations into  the following   random delayed partial differential equation  by a conjugation (The details of the derivation is given in Section 2).
\begin{equation}\label{2}
\displaystyle \frac{dv(t)}{dt}=\displaystyle  Av(t)-\mu v+L v_t +Lz(\theta_{t+\cdot}\omega)+f\left(v_t+z(\theta_{t+\cdot}\omega)\right)+Az(\theta_{t}\omega).
\end{equation}
For the above obtained pathwise deterministic delayed equation \eqref{2}, our method to overcome the obstacle caused by the lack of a variation of  constants formula is to take $V(t)=(v_t,v(t))$ and recast the equation in  an extended Hilbert space $H=\mathcal{L}\times \mathbb{X}$ equipped with the inner product
$$
((\phi, h),(\psi, k))=\int_{-\tau}^{0}(\phi(s), \psi(s))_{\mathbb{X}} ds+(h, k)_{\mathbb{X}} \quad \text { for } \quad(\phi, h),\{\psi, k\} \in H
$$
and norm
$$
\|(\phi, h)\|=((\phi, h),(\phi, h))^{1 / 2} \quad \text { for } \quad(\phi, h) \in H.
$$
Therefore, \eqref{2} can be regarded as a random partial differential equation on $H$ for which  an appropriate variation of  constants formula holds and we  investigate  the existence and smoothness of  stable and unstable invariant manifolds by the Lyapunov-Perron method established in  \cite{DLS04}. Then, the invariant manifolds of the original equations can be obtained  by a corresponding projection and the inverse random transformation.

 The  study of delay equations on the state space $X \times L^p([-\tau,0], X)$ with $X=\mathbb{C}^n$ went back to the early work of   Coleman and Mizel \cite{CM} and was then developed by many authors. See the monograph \cite{BP} for a comprehensive summary.  This  idea of recasting functional differential equations in product space has been adopted by Mohammed and Scheutzow to study the Lyapunov spectrum and stable manifolds theorem for linear or nonlinear SDDEs in \cite{M90,MS4, MS2003,MS2004}. The equations they studied do not have spatial structure and  they obtained the existence of invariant manifolds by the multiplicative ergodic theory established in \cite{RD}. Our work makes an attempt  to extend the idea of recasting retarded SPDEs in the auxiliary Hilbert space to establish the Lyapunov Perron method for analyzing existence and smoothness of invariant manifolds. Since we establish  the  existence and smoothness of   invariant manifolds for the auxiliary equation \eqref{2.4} (see Section 3) and then obtain the invariant manifold of the original equation by projection, our method successfully avoid variation of  constants formula of SPDEs with finite delays, which is also effective for deterministic PFDEs. Compared with the Lyapunov Perron method established in \cite{MS2003} and  \cite{MS2004} for SPDEs without delay, the exponential dichotomy of the semigroup generated by the linear part of the auxiliary equation \eqref{2.4} is quite different, and the construction of invariant manifolds for the original equation \eqref{1} from the auxiliary equation \eqref{2.4} needs much effort.   Moreover, recasting  stochastic or deterministic PFDEs in Hilbert space also facilitates one to tackle the problems requiring Hilbert geometric structure, such as the topological dimensions estimation of attractors.

The remaining part of this paper is organized as follows. In Section 2, we first introduce the theory of RDSs  as well as random invariant manifolds and then recast  \eqref{2} into the Hilbert space $H$, which facilitate us to treat \eqref{2} as a SPDE without delay on $H$ that generates a RDS. In Section 3, we show the existence of Lipschitz continuous stable  manifolds by the Lyapunov-Perron method.  In Section 4, we prove the smoothness of the stable manifolds as well as the existence and smoothness of  unstable manifolds. In Section 5, we apply the results to a stochastic single-species population model.

\section{Preliminaries and the auxiliary equation}
 In this section, we first introduce the concept of  RDSs following \cite{AR98} and the concept of random invariant manifolds following \cite{DLS04}. Moreover, we show  the auxiliary random partial differential equation (RPDE)  generating a RDS.
\begin{defn}\label{defn1}
Let $\left\{\theta_{t}: \Omega \rightarrow \Omega, t \in \mathbb{R}\right\}$ be a family of measure preserving transformations such that $(t, \omega) \mapsto \theta_{t} \omega$ is measurable and $\theta_{0}=\mathrm{id}$, $\theta_{t+s}=\theta_{t} \theta_{s},$ for all $s, t \in \mathbb{R}$. The flow $\theta_{t}$ together with the probability space $\left(\Omega, \mathcal{F}, P,\left(\theta_{t}\right)_{t \in \mathbb{R}}\right)$ is called a metric dynamical system.
\end{defn}
 For a given separable complete metric space $(X, \|\cdot\|_X)$,  denote by $\mathcal{B}(X)$ the  Borel-algebra of open subsets in $X$.
\begin{defn}\label{defn2}
A mapping $\Phi: \mathbb{R}^{+} \times \Omega \times X   \rightarrow X$ is said to be a random dynamical system (RDS) on a complete separable metric space  $(X,d)$ with Borel  $\sigma$-algebra  $\mathcal{B}(X)$ over the metric dynamical system $\left(\Omega, \mathcal{F}, P,\left(\theta_{t}\right)_{t \in \mathbb{R}}\right)$ if \\
(i) $\Phi(\cdot, \cdot, \cdot): \mathbb{R}^{+} \times \Omega \times X   \rightarrow X$ is $(\mathcal{B}(\mathbb{R}^{+})\times \mathcal{F}\times\mathcal{B}(X), \mathcal{B}(X))$-measurable;\\
(ii) $\Phi(0, \omega,\cdot)$ is the identity on  $X$ for $P$-a.e. $\omega \in \Omega$;\\
(iii) $\Phi(t+s, \omega,\cdot)=\Phi(t, \theta_{s} \omega,\cdot) \circ \Phi(s, \omega,\cdot),   \text { for all } t, s \in \mathbb{R}^{+}$ for $P$-a.e. $\omega \in \Omega$.\\
A RDS $\Phi$ is continuous or differentiable if $\Phi(t, \omega,\cdot): X \rightarrow X$ is continuous or differentiable for all $t\in \mathbb{R}^+$ and $P$-a.e. $\omega \in \Omega$.
\end{defn}

In this paper, we consider the canonical probability space $(\Omega, \mathcal{F}, P)$  with
$$
\Omega=\left\{\omega=\left(\omega_{1}, \omega_{2}, \ldots, \omega_{m}\right) \in C\left(\mathbb{R} ; \mathbb{R}^{m}\right): \omega(0)=0\right\}
$$
and $\mathcal{F}$ being the Borel $\sigma$-algebra induced by the compact open topology of $\Omega,$ while $P$ being the corresponding Wiener measure on $(\Omega, \mathcal{F})$. Then, we identify $W(t)$ with
$$
W(t,\omega)\equiv \left(\omega_{1}(t), \omega_{2}(t), \ldots, \omega_{m}(t)\right)\quad \text { for } t \in \mathbb{R}.
$$
and the time shift by $$\theta_{t} \omega(\cdot)=\omega(\cdot+t)-\omega(t), t \in \mathbb{R}.$$
It follows from Definition \ref{defn1} that  $\left(\Omega, \mathcal{F}, P,\left(\theta_{t}\right)_{t \in \mathbb{R}}\right)$ is a metric dynamical system.

In the sequel, we follow the idea of \cite{DLS03} to transform  \eqref{1} into a pathwise deterministic equation. The same idea has been adopted by many authors to deal with invariant manifolds for various stochastic evolution equations, such as \cite{DLS04,LS07}.
Consider the stochastic stationary solution of the one dimensional Ornstein-Uhlenbeck equation
\begin{equation}\label{2.1}
\mathrm{d} z_{j}+\mu z_{j} \mathrm{d} t=\mathrm{d} w_{j}(t), j=1, \ldots, m,
\end{equation}
which is given by
\begin{equation}\label{2.2}
z_{j}(t) \triangleq z_{j}\left(\theta_{t} \omega_{j}\right)=-\mu \int_{-\infty}^{0} e^{\mu s}\left(\theta_{t} \omega_{j}\right)(s) \mathrm{d} s, \quad t \in \mathbb{R}.
\end{equation}
Putting $z\left(\theta_{t} \omega\right)=\sum_{j=1}^{m} g_{j} z_{j}\left(\theta_{t} \omega_{j}\right)$, we have
$$
\mathrm{d} z+\mu z \mathrm{d} t=\sum_{j=1}^{m} g_{j} \mathrm{d} w_{j}.
$$
 Take the transformation $v(t)=u(t)-z\left(\theta_{t} \omega\right)$. Then,  simple computations imply
\begin{equation}\label{2.3}
\displaystyle \frac{dv(t)}{dt}=\displaystyle  Av(t)-\mu v(t)+L v_t +Lz(\theta_{t+\cdot}\omega)+f\left(v_t+z(\theta_{t+\cdot}\omega)\right)+Az(\theta_{t}\omega).
\end{equation}
Next we consider the linear part of  \eqref{2.3}, i.e. the following linear  partial functional differential equation
\begin{equation}\label{2.3b}
\displaystyle \frac{d\tilde{v}(t)}{dt}=\displaystyle  A\tilde{v}(t)-\mu \tilde{v}(t)+L \tilde{v}_t.
\end{equation}
For the linear operator $L: \mathcal{L} \mapsto \mathbb{X}$ in \eqref{2.3}, define $L_\lambda \in \mathcal{L}(\mathbb{X})$ by
 \begin{equation}\label{2.3f}
\displaystyle L_\lambda x:=L\left(e^{\lambda \cdot} x\right)
\end{equation}
for $x \in \mathbb{X}$. Throughout the remaining part of this paper, we always make the following assumptions on $A$ and $L$.

 $\mathbf{Hypothesis \    A1.}$ $A: D(A) \subset \mathbb{X} \rightarrow \mathbb{X}$ is a densely defined linear operator  that  generates a strongly continuous compact semigroup $S(t)$ on $\mathbb{X}$. $L: \mathcal{L} \mapsto \mathbb{X}$ is a bounded linear operator.

By Theorem 2.4(iii) in \cite{WJ}, \eqref{2.3b} admits a global  classical solution $\tilde{v}(t)$ for any $t\geq 0$.
Set
$$
\tilde{V}: t \mapsto\left(\begin{array}{c}
\tilde{v}_t\\
\tilde{v}(t)
\end{array}\right) \in H.
$$
Then, it follows from Corollary 3.5 in \cite{BP} that  $\tilde{V}:\mathbb{R}_{+}\rightarrow H$ is continuously differentiable with derivative
$$
\dot{\tilde{V}}(t)=\tilde{A} \tilde{V}(t),
$$
where
$$
\tilde{A}:=\left(\begin{array}{cc}
\frac{d}{d t} & 0 \\
L & A-\mu I
\end{array}\right),
$$
with domain
$$
\begin{aligned}
D(\tilde{A})=\{(\phi, h) \in H: \phi & \text { is differentialble on }[-\tau, 0],  \dot{\phi} \in \mathcal{L} \text { and } h=\phi(0) \in D(A) \}.
\end{aligned}
$$
Denote by
$$
\tilde{A}_0:=\left(\begin{array}{cc}
\frac{d}{d t} & 0 \\
0 & A-\mu I
\end{array}\right)
$$
with domain $D(\tilde{A}_0)=D(\tilde{A})$. It follows from $\mathbf{Hypothesis \    A1}$, Lemma 3.6, Theorem 3.25 in \cite{BP} that the operator $(\tilde{A}_0,D(\tilde{A}_0))$ is closed and densely defined on $H$ and generates a strongly continuous semigroup $T(t)$ given by
$$
T(t):=\left(\begin{array}{cc}
T_0(t) & S_t \\
0 & S(t)
\end{array}\right),
$$
where $\left(T_0(t)\right)_{t \geq 0}$ is the nilpotent left shift semigroup on $\mathcal{L}$ and $S_t: \mathbb{X} \rightarrow \mathcal{L}$ is defined by
$$
\left(S_t x\right)(\xi):= \begin{cases}S(t+\xi) x & \text { if }-t<\xi \leq 0, \\ 0 & \text { if }-\tau \leq \xi \leq-t .\end{cases}
$$

To guarantee that $(\tilde{A},D(\tilde{A}))$ generates a strongly continuous semigroup, we need to impose the following assumption on $L$, which was introduced in \cite{BP}.

$\mathbf{Hypothesis \    A2.}$
There exists $q: \mathbb{R}_{+} \longrightarrow \mathbb{R}_{+}$with $\lim _{t \rightarrow 0^{+}} q(t)=0$ and
$$
\int_0^t\left\|L\left(S_s \phi+T_0(s) h\right)\right\| d s \leq q(t)\left\|\left(\begin{array}{l}
\phi \\
h
\end{array}\right)\right\|
$$
for all $\left(\begin{array}{l}\phi \\ h\end{array}\right) \in D\left(\tilde{A}_0\right)$ and $t>0$.

It follows from Theorem 3.26 and 4.11 in \cite{BP} that $(\tilde{A},D(\tilde{A}))$ generates a strongly continuous semigroup $\tilde{S}(t): H \rightarrow H$ such that
$$
\|\tilde{S}(t)\|\leqslant e^{\varrho t}, \quad t \geqslant 0,
$$
under assumptions $\mathbf{Hypothesis \    A1}$ and $\mathbf{Hypothesis \    A2}$, where $\varrho\triangleq s(A-\mu I+L_\lambda)=\sup \{\Re \lambda: \lambda \in \sigma(A-\mu I+L_\lambda)\}$ is the spectral bound of the linear operator $A-\mu I+L_\lambda$.

In order to show that $\tilde{S}(t)$ satisfies  a  pseudo exponential dichotomy condition, we further impose the following assumption. \\
$\mathbf{Hypothesis \    A3.}$   Assume there exist  $\beta<\alpha$ such that the spectrum of  $A$ and $L$ satisfy
\begin{equation}\label{3a}
\begin{aligned}
\sigma(A-\mu I+L_\lambda)\cap(\beta,\alpha)=\varnothing,
\end{aligned}
\end{equation}
where $L_\lambda$ is defined by \eqref{2.3f}.

Now we prove an exponential  dichotomy result. It says that, if we assume the same compactness on the semigroup generated by $\tilde{A}, D(\tilde{A})$ as before, then the space $H$ naturally decomposes into two subspaces that are invariant under the delay semigroup.

By similar methods in \cite[Theorem 2.3]{ZW}  and the original work \cite{D81}, we show in the following  that under $\mathbf{Hypothesis \    A3}$ the semigroup $\tilde{S}(t)$ satisfies  a  pseudo exponential dichotomy condition.
 \begin{lem} \label{lem3.1}
Assume that $\mathbf{Hypothesis \    A3}$ holds, then there exists a continuous projection $P^{u}$ on $H$ such that\\
(i) $\quad P^{u} \tilde{S}(t)=\tilde{S}(t)P^{u}$;\\
(ii) the restriction $ \tilde{S}(t)|_{H^{u}}: H^{u}\rightarrow H^{u}, t \geq 0$  is an isomorphism.\\
(iii)
\begin{equation}\label{3}
\begin{aligned}
\left\|\tilde{S}(t)P^{u} x\right\| & \leq K e^{\alpha t}\|x\|, & & t \leqslant 0, \\
\left\|\tilde{S}(t)P^{s} x\right\| & \leq K e^{\beta t}\|x\|, & & t \geq 0,
\end{aligned}
\end{equation}
where $P^{s}=I-P^{u}$, $H^{s}=P^{s} H$ and $H^{u}=P^{u}H.$ Obviously, $H =H ^{u} \oplus$ $H ^{s}$. We will call $H^{s}$ and $H^{u}$ the stable subspace and the unstable subspace of $\tilde{S}(t)$  respectively.
\end{lem}
\begin{proof}
By Proposition 3.19 in \cite{BP}, $\tilde{A}$ has compact resolvent and $\sigma(\tilde{A})=\sigma(A-\mu I+L_\lambda)$.  As proved in Theorem 4.15  the
operator $(\tilde{A}, D(\tilde{A}))$ has pure point spectrum with finite-dimensional spectral subspaces under $\mathbf{Hypothesis \    A1-A3}$. We set
$\sigma_{+}(\tilde{A})=\sigma(\tilde{A}) \cap\{\lambda \in \mathbb{C} \mid \operatorname{Re} \lambda>\alpha\}, \sigma_{-}(\tilde{A})=\sigma(\tilde{A}) \cap\{\lambda \in \mathbb{C} \mid \operatorname{Re} \lambda<\beta\}.$
Obviously
$$
\sigma(\tilde{A})=\sigma_{+}(\tilde{A}) \cup \sigma_{-}(\tilde{A}), \sigma_{+}(\tilde{A}) \cap \sigma_{-}(\tilde{A})=\phi_{\text {. }}
$$
Since $\sigma_{+}(\tilde{A})$ is bounded, and both $\sigma_{+}(\tilde{A})$ and $\sigma_{-}(A) \cup\{\infty\}$ are close, Theorem 1.5.2 in \cite{D81} implies that there are projections $H^u$ and $H^s$ such that
$$
H=H^{u} \oplus H^{s},
$$
where $H^{u}=P^{u}H$ and $H^{s}=P^{s}H$ are $\tilde{A}$-invariant subspaces. Since $A$ is sectorial operator, one can check that $\tilde{A}$ is also a sectorial operator and hence it follows from Theorem 1.5.3 in \cite{D81} that statement (ii) holds.
\end{proof}

As defined in \cite{DLS03}, a multifunction $M=\{M(\omega)\}_{\omega \in \Omega}$ of nonempty closed sets $M(\omega), \omega \in \Omega$, contained in a complete separable metric space $\left(\mathcal{L}, d_{H}\right)$ is called a random set if
$$
\omega \rightarrow \inf _{y \in M(\omega)} \mathrm{d}_H(x, y)
$$
is a random variable for any $x \in H$.

A random set $M(\omega)$ is called an invariant set for a RDS $\Phi(t, \omega, x)$ if we have
$$
\Phi(t, \omega, M(\omega)) \subset M\left(\theta_{t} \omega\right) \text { for } t \geqslant 0.
$$

Moreover, if   $M$ can be represented by a graph of a $C^{k}$ (or Lipschitz) mapping
$$
h^{s}(\cdot, \omega): H^{s} \rightarrow H^{u}
$$
such that
$$
M(\omega)=M^{s}(\omega)=\left\{\xi+h^{s}(\xi, \omega) \mid \xi \in H^{s}\right\},
$$
then, we call $M^{s}(\omega)$ a $C^{k}$ (or Lipschitz) stable manifold. Here, $H^{s}$ is the stable subspace and $H^{u}$ is the unstable subspace of \eqref{2}, which are introduced in the paragraph following \eqref{3}. Conversely, if   $M$ can be represented by a graph of a $C^{k}$ (or Lipschitz) mapping
$$
h^{u}(\cdot, \omega): H^{u} \rightarrow H^{s}
$$
such that
$$
M(\omega)=M^{u}(\omega)=\left\{\xi+h^{u}(\xi, \omega) \mid \xi \in H^{u}\right\},
$$
then, we call $M^{u}(\omega)$ a $C^{k}$ (or Lipschitz) unstable manifold.

\section{Lipschitz stable  manifolds}
This section is devoted to obtaining  the existence of a Lipschitz continuous stable manifold for retarded SPDE \eqref{1}. Our main idea is first to show the existence of Lipschitz stable  manifolds for \eqref{2} by investigating \eqref{2.4} and then obtain Lipschitz stable  manifolds of \eqref{1} by the inverse conjugation.

To reformulate \eqref{2.3} as a random  partial differential equation in the Hilbert space $H$, we take $V(t)=(v_t, v(t))^T$, $\tilde{f}(t, \theta_{t}\omega,V(t))\triangleq Lz(\theta_{t+\cdot}\omega)+f\left(v_t+z(\theta_{t+\cdot}\omega)\right)+Az(\theta_{t}\omega)$ and $F(t, \theta_{t}\omega,V(t))=(0, \tilde{f}(t,\theta_{t}\omega,V(t)))$. We consider the following auxiliary random partial differential equation on $H$.
\begin{equation}\label{2.4}
\left\{\begin{array}{l}
\displaystyle \frac{dV(t)}{dt}=\displaystyle  \tilde{A}V(t)+F(t, \theta_{t}\omega, V(t)),\\
V(0)=(\phi,h), \quad(\phi, h) \in H,
\end{array}\right.
\end{equation}

Throughout the remaining part of this paper, we always make the following assumptions on  the nonlinear term $f$:

$\mathbf{Hypothesis \  A4}$ $f$ is Lipschitz continuous with $\mathbf{0}$ being a fixed point, that is, $f(\mathbf{0})=\mathbf{0}$ and $\|f(\phi)-f(\varphi)\|_{\mathbb{X}}\leq L_f\|\phi-\varphi\|_{\mathcal{L}}$ for any $\phi, \varphi\in \mathcal{L}$.

Since for $P$-a.e. $\omega \in \Omega$, \eqref{2.4} is a path-wise deterministic equation, it follows from Pazy \cite{PA}  that \eqref{2.4} admits a global mild solution which can be represented by a integral equation based on the variation of  constants formula.

\begin{lem} \label{lem2.1}
Assume that $\mathbf{Hypothesis \    A1-A4}$ hold. Then, for each $(\phi, h) \in H$, there exists a continuous function $V(\cdot,\omega, (\phi, h)):[0, \infty) \rightarrow H$ such that
\begin{equation}\label{2.5}
V(t,\omega, (\phi, h))=\tilde{S}(t)(\phi, h)+\int_{0}^{t} \tilde{S}(t-s)F(s, \theta_{s}\omega, V(s,\omega, (\phi, h))) d s, \quad t \geqslant 0
\end{equation}
for $P$-a.e. $\omega \in \Omega$.
Moreover, if $(\phi, h) \in D(\tilde{A})$, then $V(t,\omega, (\phi, h))$ is a strong solution of  \eqref{2.4} for $P$-a.e. $\omega \in \Omega$.
\end{lem}

By Theorem 3.1 and Proposition 3.2 in  \cite{SW91}, we have the following  relationship between solutions of \eqref{2.3} and \eqref{2.4}.
\begin{lem} \label{lem2.3}
Let $P_{1}$ and $P_{2}$ be the projections of $H$ onto $\mathcal{L}$ and $\mathbb{X}$ respectively. Then, the following statements are true.\\
(i) $T(t) h=P_{2} \tilde{S}(t)(\phi, h)$ for all $(\phi, h) \in H$, where $T(t)$ is the semigroup generated by $A-\mu I$ on $X$.\\
(ii) For every $\phi \in \mathcal{L}$, define $v(\cdot, \omega, \phi):[-\tau, \infty) \rightarrow \mathbb{X}$ by
$$
v(t, \omega, \phi)= \begin{cases}\phi(t) & -\tau \leqslant t<0 \\ P_{2} V(t,\omega, (\phi, \phi(0)))& t \geqslant 0\end{cases}
$$
Then $v_t(\cdot, \omega, \phi)=P_{1} V(t,\omega, (\phi, \phi(0)))$ for $t \geqslant 0$. Moreover, if $(\phi, h) \in D(\tilde{A})$, then $v_t(\cdot, \omega, \phi)$ is the strong solution to \eqref{2.3} with initial condition $v_0=\phi$ for  $t \geqslant 0$ and $P$-a.e. $\omega\in \Omega$.
\end{lem}

For the constants $\alpha, \beta$ defined in \eqref{3} and each $\eta$ such that $\beta<\eta<\alpha$, we define the Banach space
$C_{\eta}^{+}=\left\{\phi(\cdot):[0, \infty) \rightarrow H\mid \phi\right.$ is continuous and $\left.\sup _{t \in[0, \infty)} e^{-\eta t}\|\phi\|<\infty\right\}$ and $P_1C_{\eta}^{+}=\left\{\varphi=P_1\phi: \phi\in C_{\eta}^{+}\right\}$
with the norms
$$
|\phi|_{C_{\eta}^{+}}=\sup _{t \in[0, \infty)} e^{-\eta t }\|\phi\|,
$$
and
$$
|\varphi|_{P_1C_{\eta}^{+}}=\sup _{t \in[0, \infty)} e^{-\eta t }\|\varphi\|_{\mathcal{L}}.
$$
Denote by
$$
N^{s}(\omega)=\left\{\psi \in H  \mid V\left(t, \omega, \psi\right) \in C_{\eta}^{+}\right\}.
$$
By the definition of $C_{\eta}^{+}$, we can see that $N^{s}(\omega)$ consists of all  initial datum through which solutions are controlled by $e^{\eta t}$. In the sequel, we obtain the Lipschitz stable invariant manifolds of \eqref{2.3} based on $N^{s}(\omega)$.
\begin{thm}\label{thm3.1} Assume that $A, L, f$  satisfy  $\mathbf{Hypothesis \    A1-A3}$ with the constants $K, \alpha, \beta$ defined in \eqref{3}. Moreover, assume that  $L_f$ satisfy
$$
K L_f \left(\frac{1}{\eta-\beta}+\frac{1}{\alpha-\eta}\right)<1.
$$
Then, the delayed random partial differential  equation \eqref{2.3} adimits a Lipschitz stable  invariant  manifold $M^{s}(\omega)=P_1 N^{s}(\omega)$, which can be represented by
$$
M^{s}(\omega)=\left\{\zeta+h^{s}(\zeta, \omega)\mid \zeta\in P_1 P^{s}N^{s}(\omega)\right\}.
$$
Here, $h^{s}(\cdot,\omega): P_1 P^{s}N^{s}(\omega) \rightarrow P_1 P^{u}N^{s}(\omega)$ is a Lipschitz continuous mapping that satisfies $h^{s}(0,\omega)=0$.
\end{thm}
\begin{proof}   It follows from the proof of Theorem 3.2 in \cite{DLS04} that  $N^{s}(\omega)$ is a Lipschitz stable  invariant  manifold of the random partial differential equation \eqref{2.4} in $H$ under the condition of Theorem \ref{thm3.1}. In the sequel, we show that $M^{s}(\omega)=\left\{\zeta+h^{s}(\zeta, \omega)\mid \zeta\in P_1P^s N^{s}(\omega)\right\}$ is the Lipschitz stable  invariant stable manifold of \eqref{2.3}. It follows from the proof of Theorem 3.1 in \cite{DLS04} that   $\psi \in N^{s}(\omega)$ if and only if there exists a function $V(\cdot) \in C_{\eta}^{+}$ with $V(0)=(\phi, \phi(0))\triangleq\psi$ such that
\begin{equation}\label{3.2}
\begin{aligned}
V(t) =&  \tilde{S}(t)\hat{\zeta}+\int_{0}^{t}P^{s} \tilde{S}(t-s)F(s, \theta_s\omega, V(s))d s+\int_{\infty}^{t} P^{u}  \tilde{S}(t-s)  F(s,\theta_s\omega, V(s)) ds, \\
\end{aligned}
\end{equation}
where $\hat{\zeta}=P^s \psi\in H^s$.
By Lemma \ref{lem2.3}, \eqref{3.2} is equivalent to
\begin{eqnarray}\label{3.3}
v_t(\cdot,\omega, \phi)&=&P_1V(t,\omega,\psi) \\
&=&P_1 \tilde{S}(t)\hat{\zeta}+\int_{0}^{t}P_1P^{s} \tilde{S}(t-s)F(s,\theta_s\omega, P_1^{-1}v_s)d s+\int_{\infty}^{t}P_1P^{u} \tilde{S}(t-s)  F(s,\theta_s\omega, P_1^{-1}v_s) ds.\notag
\end{eqnarray}
Hence, $\psi \in N^{s}(\omega)$ if and only if there exists a function $v_t \in P_1C_{\eta}^{+}$ with $v_0=\phi$ such that \eqref{3.3} holds.
Setting $t=0$ in  \eqref{3.3} we have
\begin{equation}\label{3.3a}
\begin{aligned}
v_0(\cdot,\omega, \phi)=&P_1V(0,\omega,\psi) \\
=&P_1\hat{\zeta}+\int_{\infty}^{0}P_1P^{u} \tilde{S}(-s)  F(s,\theta_s\omega, P_1^{-1}v_s) ds \\
=&\zeta+h^{s}(\zeta, \omega),
\end{aligned}
\end{equation}
where $\zeta=P_1P^s \psi \in P_1P^s N^s$ and $h^{s}(\zeta, \omega)$ is defined by
$$
h^{s}(\zeta, \omega)=\int_{\infty}^{0}P_1 P^{u} \tilde{S}(-s)  F(s,\theta_s\omega,P_1^{-1}v_s(\cdot,\omega,\zeta))\mathrm{d} s
$$
satisfying $h^{s}(0, \omega)=0$. Moreover, by the same arguments as the proof Theorem 3.1 in \cite{DLS04}, one can see that  $h^{s}$ is measurable.

To show that $M^{s}(\omega)$ is a random set, by  Castaing and Valadier \cite[Theorem III.9]{CV},  it suffices to prove that, for any $x \in \mathcal{L}$,
\begin{equation}\label{3.10}
\omega \rightarrow \inf _{y \in N^s(\omega)}\left\|x-\left(P_{1}P^s y+h^{s}\left(P_{1}P^s y, \omega\right)\right)\right\|
\end{equation}
is measurable. Since $N^s(\omega)$ is a random set, the measurability of $\omega \rightarrow h^{s}\left(P_{1}P^{s} y, \omega\right)$  for any $y \in N^s(\omega)$ indicates that \eqref{3.10} is measurable.

At last, we prove that $M^{s}(\omega)$ is invariant under the assumption of Theorem \ref{thm3.1}. Since $N^s(\omega)$ is invariant under the RDS $\Phi$, that is, for each $\psi\in N^s(\omega)$, $\Phi(t,\omega,\psi)=V(t,\omega,\psi)\in N^s(\theta_t\omega)$, we have for any $\chi=P_1\psi\in P_1 N^{s}(\omega)$, $\Psi(t,\omega, \chi)=v_t\left(\cdot, \omega, \chi\right)=P_1V(t,\omega,P_1^{-1}\chi)=P_1V(t,\omega,\psi)\in P_1 N^{s}(\theta_t\omega)$ for all $t \geqslant 0$.
\end{proof}

\begin{thm}\label{thm2.1}
Assume that $M^{s}(\omega)$  is a Lipschitz stable  manifold of the random partial differential equation \eqref{2.3} on $\mathcal{L}$, then $\tilde{M}^{s}(\omega)=\left\{\psi(\xi)+z(\theta_\cdot \omega) \mid \psi \in {M}^{s}(\omega)\right\}$ is a Lipschitz stable  manifold of the stochastic retarded partial differential equation \eqref{1} on $\mathcal{L}$.
\end{thm}
\begin{proof}We first prove $\tilde{M}^{s}(\omega)$ is invariant under $u_s$, i.e., for each $\phi \in \tilde{M}^{s}(\omega)$, $u_s\left(\cdot, \phi, \omega\right) \in \tilde{M}^{s}\left(\theta_{s} \omega\right)$ for all $s \geqslant 0 .$ By the relationship $u_{s}\left(\cdot, \psi,\omega\right)=v_{s}\left(\cdot, \psi,\omega\right)+z(\theta_{s+\cdot}\omega)$, the stationary property of $z(\theta_{s+\cdot}\omega)$ and $v_{s}\left(\cdot, \psi,\omega\right)\in {M}^{s}(\theta_s\omega)$, one can see that $u_s\left(\cdot, \psi, \omega\right) \in \tilde{M}^{s}\left(\theta_{s} \omega\right)$.
In the following, we prove that $\tilde{M}^{s}(\omega)$ can be represented by a Lipschitz function. For $z(\theta_\cdot \omega)\in \mathcal{L}=P^1 H$, we split it as $z(\theta_\cdot \omega)=z_1+z_2$ with $z_1\in P^1 P^s H$ and $z_2\in P^1 P^u H$. Apparently, $z_2=P^1 P^u (P^s)^{-1} z_1$.
By the definition, we have
$$
\begin{array}{l}
\tilde{M}^{s}(\omega)\triangleq \left\{\psi +z(\theta_\cdot \omega) \mid \psi \in {M}^{s}(\omega)\right\} \\
=\left\{z_1+\zeta+h^{s}(\zeta, \omega)+z_2 \mid \zeta \in P^1P^{s}N^s(\omega)\right\} \\
=\left\{\bar{\zeta}+\hat{h}^{s}(\bar{\zeta}, \omega) \mid\bar{ \zeta} \in  P^1P^{s}N^s(\omega)\right\}, \\
\end{array}
$$
where $\bar{\zeta}=z_1+\zeta$ and $\hat{h}^{s}(\bar{\zeta}, \omega)=h^{s}(\bar{\zeta}-z_1, \omega)+P^1 P^u (P^s)^{-1} z_1$. It follows from the Lipschitz continuity of $h$ that $\hat{h}^{s}$ is Lipschitz,  implying that $\tilde{M}^{s}(\omega)$ is a Lipschitz stable manifold given by the graph of a Lipschitz continuous function $\hat{h}^{s}(\bar{\zeta}, \omega)$ over the space $P^1P^{s}N^s(\omega)$.
\end{proof}

Next we will  prove the $C^{k}$ smoothness of the  Lipschitz stable manifold $M^{s}(\omega)$  for each $\omega\in \Omega$.

 \begin{thm}\label{thm4.1} Assume that $A, L, f$  satisfy  $\mathbf{Hypothesis \    A1-A3}$ with the constants $K, \alpha, \beta$ defined in \eqref{3}. Moreover, assume  that  $f: \mathcal{L}\rightarrow \mathbb{X}$ is $C^{k}$. If $\beta<k \eta<\alpha$ and
$$
K L_f\left(\frac{1}{i \eta-\beta}+\frac{1}{\alpha-i \eta}\right)<1 \quad \text { for all } 1 \leqslant i \leqslant k,
$$
then the invariant stable manifold $M^{s}(\omega)$ for the delayed random partial differential equation  \eqref{2.3}  is   $C^{k}$, that is, $h(\zeta, \omega)$ is $C^{k}$ in $\zeta$.
\end{thm}
\begin{proof}
It follows from the proof of Theorem 3.2 in \cite{DLS04} that  $N^{s}(\omega)$ is a $C^k$ smoothness invariant stable manifold of the random partial differential equation \eqref{2.4} in $H$ under the condition of Theorem \ref{thm4.1}. Now we show that $$
M^{s}(\omega)=\left\{\zeta+h^{s}(\zeta, \omega) \mid \zeta \in P_1P^sN^{s}\right\}
$$
is the smooth invariant stable manifold of \eqref{2.3}.  The invariance has been proved in Theorem \ref{thm3.1} and we  will  prove the $C^k$ smoothness by an induction method. In the case  $k=1$, because of
$$
K L_f\left(\frac{1}{\eta-\beta}+\frac{1}{\alpha-\eta}\right)<1,
$$
the continuity and monotonicity of $g(x)=\frac{1}{x-\beta}+\frac{1}{\alpha-x}$ implies that there must be a small number $\sigma>0$ such that if  $\beta+\sigma<\eta$, then
$$
K L_f \left(\frac{1}{(\eta-\delta)-\beta}+\frac{1}{\alpha-(\eta-\delta)}\right)<1 \quad \text { for all } 0 \leqslant \delta \leqslant  \sigma.
$$
Denote by $I^{s}(v_t, \hat{\zeta})$  the right hand side of equality \eqref{3.3}. Apparently,  $I^{s}:P_1C_{\eta}^{+} \times P_1H^{s}\rightarrow C_{\eta}^{+}$ is well-defined. Moreover,  for any given $v_t, \bar{v_t} \in P_1 C_{\eta}^{+}$, we can obtain
\begin{equation}\label{3.7}
\begin{aligned}
\left|I^{s}(v_t, \zeta)-I^{s}(\bar{v_t}, \zeta)\right|_{P_1C_{\eta}^{+}}
& \leqslant \sup _{t \in[0, \infty)}e ^ { - \eta t } (\mid \int_{0}^{t}P_1U(t-s)  P^{s}  [F(s,\theta_s\omega,P_1^{-1}v_s)- F(s,\theta_s\omega,P_1^{-1}v_s)] \mathrm{d}s \\
& +\int_{\infty}^{t}P_1U(t-s)  P^{u} [ F(s,\theta_s\omega,P_1^{-1}v_s)- F(s,\theta_s\omega,P_1^{-1}v_s)] \mathrm{d}s\mid)\\
& \leqslant \sup _{t \in[0, \infty)}\left\{K L_f |v_s-\bar{v_s}|_{P_1C_{\eta}^{+}}\left(\int_{0}^{t} e^{(\beta-\eta)(t-s)} d s+\int_{t}^{\infty} e^{(\alpha-\eta)(t-s)} \mathrm{d} s\right)\right\} \\
& \leqslant K L_f  \left(\frac{1}{\eta-\beta}+\frac{1}{\alpha-\eta}\right)|v_s-\bar{v_s}|_{P_1C_{\eta}^{+}}.
\end{aligned}
\end{equation}
Clearly, $I^{s}$ is Lipschitz continuous in $\zeta$, since for any $\zeta_1, \zeta_2\in P_1C_{\eta}^{+}$, $\left|I^{s}(v_t, \zeta)-I^{s}(\bar{v_t}, \zeta)\right|_{P_1C_{\eta}^{+}}
\leqslant Ke^{\beta t}\|\zeta_1-\zeta_2\|_{P_1C_{\eta}^{+}}$.
Hence,  $I^{s}:P_1C_{\eta}^{+} \times P_1H^{s}\rightarrow P_1C_{\eta}^{+}$ is well-defined and is a  contraction in $P_1C_{\eta-\sigma}^{+} \subset P_1C_{\eta}^{+}$ for any $0 \leqslant \delta \leqslant  \sigma$ by the above discussion. Hence, $v_t(\cdot, \xi, \omega) \in P_1C_{\eta-\delta}^{+}$.

For any $\psi \in P_1C_{\eta-\delta}^{+}$  and $\zeta \in P_1C_{\eta}^{+}$, we  define
$$
\begin{aligned}
S \psi=& \int_{0}^{t} P_1 \tilde{S}(t-s) P^{s} D_{v_t} F(s,\theta_s\omega,P_1^{-1}v_s(\cdot, \omega, \zeta))  \psi \mathrm{~d} s
\\&+\int_{\infty}^{t} P_1 \tilde{S}(t-s) P^{u} D_{v_t} F(s,\theta_s\omega,P_1^{-1}v_s(\cdot, \omega, \zeta)) \psi \mathrm{~d} s.
\end{aligned}
$$
By fact  that $I^{s}: P_1C_{\eta-\sigma}^{+}\rightarrow P_1C_{\eta-\sigma}^{+}$ is a contraction,  we can see that  $S: P_1C_{\eta-\sigma}^{+}\rightarrow P_1C_{\eta-\sigma}^{+}$ is bounded by the following estimate
$$
\|S\| \leqslant K L_f\left(\frac{1}{(\eta-\sigma)-\beta}+\frac{1}{\alpha-(\eta-\sigma)}\right)<1.
$$
Therefore, we can see   that $I d-T$ is invertible in $P_1C_{\eta-\sigma}^{+}$.  For any given  $\zeta, \zeta_{0} \in H^{s}$, we define
$$
\begin{array}{c}
J=\int_{0}^{t}P_1  \tilde{S}(t-s)P^{s}[ F(s,\theta_s\omega, P_1^{-1}v_s(\cdot, \omega, \zeta)-F(s,\theta_s\omega, P_1^{-1}v_s(\cdot, \omega, \zeta_0))-\\D_{v_t} F(s,\theta_s\omega, P_1^{-1}v_s(\cdot, \omega, \zeta))
\left(v(\cdot, \omega, \zeta)-v(\cdot, \omega, \zeta_0)\right)] \mathrm{d} s
+\int_{\infty}^{t}  \tilde{S}P^{u}[ F(s,\theta_s\omega, P_1^{-1}v_s(\cdot, \omega, \zeta))- \\ F(s,\theta_s\omega, P_1^{-1}v_s(\cdot, \omega, \zeta_0))-  D_{v_t} F(s,\theta_s\omega, P_1^{-1}v_s(\cdot, \omega, \zeta))\left(v_s(\cdot, \omega, \zeta)-v_s(\cdot, \omega, \zeta_0)\right)] \mathrm{d} s.
\end{array}
$$
If we can prove that $|J|_{P_1C_{\eta-\sigma}^{+}}=o\left(\|\zeta-\zeta_{0}\|_{\mathcal{L}}\right)$ as $\zeta \rightarrow \zeta_{0}$, then we have
$$
\begin{aligned}
v_t(\cdot, \omega, \zeta)-v_t\left(\cdot, \zeta_{0}, \omega\right)-S\left(v_t(\cdot, \omega, \zeta)-v_t(\cdot, \omega, \zeta_0)\right)&= S\left(\zeta-\zeta_{0}\right)+J \\
&= S\left(\zeta-\zeta_{0}\right)+o\left(\|\zeta-\zeta_{0}\|_{\mathcal{L}}\right), \text { as } \zeta \rightarrow \zeta_{0},
\end{aligned}
$$
from which we can see
$$
v_t(\cdot, \omega, \zeta)-v_t(\cdot, \omega, \zeta_0)=(I d-S)^{-1} S\left(\zeta-\zeta_{0}\right)+o\left(\left|\zeta-\zeta_{0}\right|\right).
$$
Therefore, $v_t(\cdot, \omega, \zeta)$ is differentiable in $\zeta$ and its derivative satisfies $D_{\zeta} v_t(\cdot ; \zeta, \omega)$ $\in L\left(P_1H^{s}, C_{\eta-\delta}^{+}\right)$, where $L\left(P_1H^{s}, P_1C_{\eta-\delta}^{+}\right)$ is the usual space of bounded linear operators and
\begin{equation}\label{4.1}
\begin{aligned}
D_{\zeta} v_t(\cdot, \omega, \zeta)=&P_1  \tilde{S}(t)P^{s}+\int_{0}^{t}P_1  \tilde{S}(t-s) P^{s} D_{v_t} F(s,P_1^{-1}v_s(\cdot, \omega, \zeta),\theta_s\omega) D_{\zeta} v_s(\cdot, \zeta, \omega) \mathrm{d} s \\
&+\int_{\infty}^{t} P_1 \tilde{S}(t-s)  P^{u} D_{v_t} F(s,P_1^{-1}v_s(\cdot, \omega, \zeta),\theta_s\omega) D_{\zeta} v_s(\cdot, \omega, \zeta) \mathrm{d} s.
\end{aligned}
\end{equation}
In the sequel, we show  that $|J|_{P_1C_{\eta-\sigma}^{+}}=\mathrm{o}\left(\left|\zeta-\zeta_{0}\right|\right)$ as $\zeta \rightarrow \zeta_{0}$  indeed holds.  Let $N$ be an appropriate  large positive number to be chosen later and define
$$
\begin{aligned}
J_{1}=& e^{-(\eta-\sigma) t}\{\mid \int_{N}^{t}P_1  \tilde{S}(t-s) P^{s}[ F(s,P_1^{-1}v_s(\cdot, \omega, \zeta),\theta_s\omega)-P_1^{-1}F(s,v_s\left(\cdot, \omega, \zeta_{0}\right),\theta_s\omega)-\\
&D_{v_t} F(s,v_s(\cdot, \omega, \zeta_0),\theta_s\omega)\left(P_1^{-1}v_s(\cdot, \omega, \zeta)-P_1^{-1}v_s(\cdot, \omega, \zeta_{0})\right)] \mathrm{d} s \mid \}
\end{aligned}
$$
for $t \geqslant N$ and $J_{1}=0$ for $t<N$
$$
\begin{aligned}
J_{2}=& e^{-(\eta-\sigma) t}\{\mid \int_{0}^{N}P_1  \tilde{S}(t-s) P^{s}[ F(s,P_1^{-1}v_s(\cdot, \omega, \zeta),\theta_s\omega)-  F(s,P_1^{-1}v_s(\cdot, \omega, \zeta_{0}),\theta_s\omega)-\\
&D_{v_t} F(s,P_1^{-1}v_s(\cdot, \omega, \zeta_{0}),\theta_s\omega)\left(P_1^{-1}v_s(\cdot, \omega, \zeta)-P_1^{-1}v_s(\cdot, \omega, \zeta_{0})\right)] \mathrm{d} s \mid \}.
\end{aligned}
$$

Also, for a large positive number $\bar{N}$  to be chosen later and $0 \leqslant t \leqslant \bar{N}$, let
$$
\begin{aligned}
J_{3}=& e^{-(\eta-\sigma) t}\{\mid \int_{\bar{N}}^{t} P_1 \tilde{S}(t-s) P^{u}[ F(s,P_1^{-1}v_s(\cdot, \omega, \zeta),\theta_s\omega)-  F(s,P_1^{-1}v_s(\cdot, \omega, \zeta_{0}),\theta_s\omega)-\\
&D_{v_t} F(s,P_1^{-1}v_s(\cdot, \omega, \zeta_{0}),\theta_s\omega)\left(P_1^{-1}v_s(\cdot, \omega, \zeta)-P_1^{-1}v_s(\cdot, \omega, \zeta_{0})\right)] \mathrm{d} s\mid \},
\end{aligned}
$$
$$
\begin{aligned}
J_{4}=& e^{-(\eta-\sigma) t} \{\mid \int_{\infty}^{\bar{N}}P_1 \tilde{S}(t-s) P^{u}[ F(s,P_1^{-1}v_s(\cdot, \omega, \zeta),\theta_s\omega)-  F(s,P_1^{-1}v_s(\cdot, \omega, \zeta_{0}),\theta_s\omega)-\\
&D_{v_t} F(s,P_1^{-1}v_s(\cdot, \omega, \zeta_{0}),\theta_s\omega)\left(P_1^{-1}v_s(\cdot, \omega, \zeta)-P_1^{-1}v_s(\cdot, \omega, \zeta_{0})\right)] \mathrm{d} s \mid \}
\end{aligned}
$$
In the case $t \geqslant \bar{N}$, define
$$
\begin{aligned}
J_{5}=& e^{-(\eta-\sigma) t}\int_{\infty}^{t}\{\mid P_1 \tilde{S}(t-s) P^{u}[ F(s,P_1^{-1}v_s(\cdot, \omega, \zeta),\theta_s\omega)-  F(s,P_1^{-1}v_s(\cdot, \omega, \zeta_{0}),\theta_s\omega)-\\
&D_{v_t} F(s,P_1^{-1}v_s(\cdot, \omega, \zeta_{0}),\theta_s\omega)\left(P_1^{-1}v_s(\cdot, \omega, \zeta)-P_1^{-1}v_s(\cdot, \omega, \zeta_{0})\right)] \mathrm{d} s  \mid \}
\end{aligned}
$$
By the definition, we only need to prove that for any $\varepsilon>0$ there is a $\delta>0$ such that $|J|_{P_1C_{\eta-\sigma}^{+}} \leqslant \varepsilon\left|\zeta-\zeta_{0}\right|$ provided that $\left|\zeta-\zeta_{0}\right| \leqslant \delta$. Keep in mind  that
$$
|J|_{P_1C_{\eta-\sigma}^{u}} \leqslant \sup _{t \geqslant 0} J_{1}+\sup _{t \geqslant 0} J_{2}+\sup _{0 \leqslant t \leqslant \bar{N}} J_{3}+\sup _{0 \leqslant t \leqslant \bar{N}} J_{4}+\sup _{t \geqslant \bar{N}} J_{5}.
$$
By the methods for the derivation of \eqref{3.10}, we can obtain that
$$
\begin{aligned}
J_{1} & \leqslant 2 K L_f  \int_{N}^{t} e^{(\beta-(\eta-\sigma))(t-s)} e^{-\sigma s}\left|\left(P_1^{-1}v_s(\cdot, \omega, \zeta)-P_1^{-1}v_s(\cdot, \omega, \zeta_{0})\right)\right|_{C_{\eta-\sigma}^{u}} \mathrm{~d} s \\
& \leqslant \frac{2 K^{2} L_f e^{-\sigma N}}{((\eta-\sigma)-\beta)\left(1-L_f \left(\frac{1}{(\eta-\sigma)-\beta}+\frac{1}{\alpha-(\eta-\sigma)}\right)\right.}\left|\zeta-\zeta_{0}\right| .
\end{aligned}
$$
Choose $N$  large enough such that
$$
\frac{2 K^{2} L_f e^{-\sigma N}}{(\eta-\sigma-\beta)\left(1-K L_f\left(\frac{1}{\eta- \sigma-\beta}+\frac{1}{\alpha-(\eta-\sigma)}\right)\right)} \leqslant \frac{1}{4}\varepsilon,
$$
implying that
$$
\sup _{t \geqslant 0} J_{1} \leqslant \frac{1}{4} \varepsilon \left\|\zeta-\zeta_{0}\right\|,
$$
for any $t\geq N$. Fixing such $N$, for $J_{2}$ we have that
$$
\begin{aligned}
J_{2} \leqslant & K \int_{0}^{N} e^{(\beta-(\eta-\sigma))(t-s)}\int _ { 0 } ^ { 1 } \mid D_{v_t} F(s,\tau P_1^{-1}v_s(\cdot, \omega, \zeta_{0})+(1-\tau) P_1^{-1}v_s(\cdot, \omega, \zeta_{0}),\theta_s\omega)-\\
& D_{v_t} F(s,v_s(\cdot, \zeta_0, \omega),\theta_s\omega) \mid \mathrm{d} \tau\left|\left(P_1^{-1}v_s(\cdot, \omega, \zeta)-P_1^{-1}v_s(\cdot, \omega, \zeta_{0})\right)\right|_{C_{\eta-\sigma}^{u}} \mathrm{d} s \\
\leqslant & \frac{K^{2}\left|\zeta-\zeta_{0}\right|\|P_1^{-1}\|}{1-KL_f\left(\frac{1}{\eta-\sigma-\beta}+\frac{1}{\alpha-(\eta-\sigma)}\right)} \int_{0}^{N} e^{-(\beta-(\eta-\delta)) s}
\int _ { 0 } ^ { 1 } \mid D_{v_t} F(s,\tau P_1^{-1}v_s(\cdot, \omega, \zeta_{0})+
\\&(1-\tau) P_1^{-1}v_s(\cdot, \omega, \zeta_{0}),\theta_s\omega)-D_{v_t} F(s,P_1^{-1}v_s(\cdot, \zeta_0, \omega),\theta_s\omega) \mid \mathrm{d} \tau\\
&\left|\left(P_1^{-1}v_s(\cdot, \omega, \zeta)-P_1^{-1}v_s(\cdot, \omega, \zeta_{0})\right)\right|_{C_{\eta-\sigma}^{u}} \mathrm{d} s.\\
\end{aligned}
$$
Since the last integral is on the compact interval $[0, N]$ and $D_{v_t} F(s,P_1^{-1}v_s(\cdot, \omega, \zeta),\theta_s\omega)$ is continuous with respect to $\zeta$, we can see that $\|D_{v_t} F(s,\tau P_1^{-1}v_s(\cdot, \omega, \zeta_{0})+(1-\tau) P_1^{-1}v_s(\cdot, \omega, \zeta_{0}),\theta_s\omega)-D_{v_t} F(s,v_s(\cdot, \zeta_0, \omega),\theta_s\omega)\|\rightarrow 0$ as $\|\zeta-\zeta_0\|\rightarrow 0$. Hence, there exists a $\delta_{1}>0$ such that if $\left|\zeta-\zeta_{0}\right| \leqslant \delta_{1}$, then
$$
\sup _{t \geqslant 0}J_{2} \leqslant \frac{1}{4}\varepsilon\left|\zeta-\zeta_{0}\right|.
$$
By similar arguments as the computation of $J_1$, we can  choose  sufficiently large $\bar{N}$ such that
$$
\sup _{0 \leqslant t \leqslant \bar{N}} J_{4}+\sup _{t \geqslant \bar{N}} J_{5} \leqslant \frac{1}{4} \varepsilon\left|\zeta-\zeta_{0}\right|
$$
For such fixed  $\bar{N}$, there exists $\delta_{2}>0$ such that if $\left|\zeta-\zeta_{0}\right| \leqslant \delta_{2}$, then
$$
\sup _{0 \leqslant t \leqslant \bar{N}} J_{3} \leqslant \frac{1}{4} \varepsilon\left|\zeta_{1}-\zeta_{2}\right|
$$
Taking $\delta=\min \left\{\delta_{1}, \delta_{2}\right\}$, we can see that
$$
|J|_{C_{\eta-\sigma}^{u}} \leqslant \varepsilon\left|\zeta-\zeta_{0}\right|,
$$
provided $\left|\zeta-\zeta_{0}\right| \leqslant \delta$, indicating that $|J|_{C_{\eta-\sigma}^{+}}=o\left(\left|\zeta-\zeta_{0}\right|\right)$ as $\zeta \rightarrow \zeta_{0}$. Subsequently, we show that $D_{\zeta}v_t(\cdot, \cdot, \omega): P_1H^{s}\rightarrow P_1C_{\eta}^{+}$ is continuous. For any $\zeta, \zeta_{0} \in P_1H^{s}$, using \eqref{4.1}, we have
\begin{equation}\label{4.2}
\begin{aligned}
D_{\zeta} v_t(\cdot, \omega, \zeta)-D_{\zeta} v_t(\cdot, \zeta_0, \omega)=& \int_{0}^{t} P_1 \tilde{S}(t-s) P^{s}[ D_{v_t} F(s,P_1^{-1}v_s(\cdot, \omega, \zeta),\theta_s\omega) D_{\zeta} P_1^{-1}v_s(\cdot, \omega, \zeta)-\\
& D_{v_t} F(s,v_s(\cdot, \zeta_0, \omega),\theta_s\omega) D_{\zeta} v_s(\cdot, \zeta_0, \omega)]\mathrm{d} s +\int_{\infty}^{t} \tilde{S}(t-s)  P^{u}\\
& [D_{v_t} F(s,P_1^{-1}v_s(\cdot, \omega, \zeta),\theta_s\omega) D_{\zeta} P_1^{-1}v_s(\cdot, \omega, \zeta)- D_{v_t} F(s,v_s(\cdot, \zeta_0, \omega),\theta_s\omega) \\
&D_{\zeta} v_s(\cdot, \zeta_0, \omega)]\mathrm{d} s\\
=& \int_{0}^{t}P_1  \tilde{S}(t-s) P^{s}D_{v_t} F(s,P_1^{-1}v_s(\cdot, \omega, \zeta),\theta_s\omega) [D_{\zeta} P_1^{-1}v_s(\cdot, \omega, \zeta)-\\
& D_{\zeta} v_s(\cdot, \zeta_0, \omega)]\mathrm{d} s +\int_{\infty}^{t} P_1 \tilde{S}(t-s)  P^{u}D_{v_t} F(s,P_1^{-1}v_s(\cdot, \omega, \zeta),\theta_s\omega)\\
& [ D_{\zeta} P_1^{-1}v_s(\cdot, \omega, \zeta)- D_{\zeta} v_s(\cdot, \zeta_0, \omega)]\mathrm{d} s+\bar{J},
\end{aligned}
\end{equation}
where
\begin{equation}\label{4.3}
\begin{aligned}
\bar{J}=& \int_{0}^{t}P_1  \tilde{S}(t-s) P^{s}[D_{v_t} F(s,P_1^{-1}v_s(\cdot, \omega, \zeta),\theta_s\omega)-D_{v_t} F(s,v_s(\cdot, \zeta_0, \omega),\theta_s\omega)] D_{\zeta} v_s(\cdot, \zeta_0, \omega)\mathrm{d} s \\
&+\int_{\infty}^{t}P_1  \tilde{S}(t-s)  P^{u} [D_{v_t} F(s,P_1^{-1}v_s(\cdot, \omega, \zeta),\theta_s\omega) - D_{v_t} F(s,v_s(\cdot, \zeta_0, \omega),\theta_s\omega)]D_{\zeta} v_s(\cdot, \zeta_0, \omega)\mathrm{d} s.
\end{aligned}
\end{equation}
Therefore, we have
$$
\left|D_{\zeta} v_t(\cdot, \omega, \zeta)-D_{\zeta} v_t(\cdot, \zeta_0, \omega)\right|_{L\left(P_1H^{s}, P_1C_{\eta}^{+}\right)}
\leqslant \frac{|\bar{J}|_{L\left(P_1H^{s}, P_1C_{\eta}^{+}\right)}}{1-K L_f\left(\frac{1}{\eta-\beta}+\frac{1}{\alpha-\eta}\right)}.
$$
By similar argument as we used for the estimation of $|J|$, one can also see that $|\bar{J}|_{L\left(P_1H^{s}, P_1C_{\eta}^{+}\right)}=o(1)$ as $\zeta \rightarrow \zeta_{0}$. Hence $D_{\xi} v_t(\cdot, \cdot, \omega)$ is continuous from $P_1H^{s}$ to
$L\left(P_1H^{s}, P_1C_{\eta}^{+}\right) .$ Therefore, $v_t(\cdot ; \cdot, \omega)$ is $C^{1}$ from $P_1H^{s}$ to $P_1C_{\eta}^{+}$. Now we show that
$u$ is $C^{k}$ from $P_1H^{s}$ to $P_1C_{k \eta}^{+}$ by induction for $k \geqslant 2 .$ By the induction assumption, we know that $u$ is $C^{k-1}$ from $P_1H^{s}$ to $P_1C_{(k-1) \eta}^{+}$ and the  $(k-1)$ derivative $D_{\xi}^{k-1} u(t ; \xi, \omega)$ satisfies the following equation.
$$
\begin{aligned}
D_{\xi}^{k-1} u=& \int_{0}^{t}P_1  \tilde{S}(t-s) P^{s}D_{v_t} F(s,P_1^{-1}v_s(\cdot, \omega, \zeta),\theta_s\omega) D^{k-1}_{\zeta} P_1^{-1}v_s(\cdot, \omega, \zeta) \mathrm{~d} s \\
&+\int_{\infty}^{t}P_1 \tilde{S}(t-s) P^{u}D_{v_t} F(s,P_1^{-1}v_s(\cdot, \omega, \zeta),\theta_s\omega) D^{k-1}_{\zeta} P_1^{-1}v_s(\cdot, \omega, \zeta)  \mathrm{~d} s \\
&+\int_{0}^{t}P_1  \tilde{S}(t-s) P^{s} R_{k-1}(s, \zeta, \omega) \mathrm{d} s+\int_{\infty}^{t}P_1  \tilde{S}(t-s) P^{u} R_{k-1}(s, \zeta, \omega) \mathrm{d} s,
\end{aligned}
$$
where
$$
R_{k-1}(s, \zeta, \omega)=\sum_{i=0}^{k-3}\left(\begin{array}{c}
k-2 \\
i
\end{array}\right) D_{\zeta}^{k-2-i}\left(D_{v_t} F(s,P_1^{-1}v_s(\cdot, \omega, \zeta),\theta_s\omega)\right) D_{\zeta}^{i+1} u(s ; \xi, \omega)
$$
By the induction hypothesis, we have that $D_{\zeta}^{i} v_t(\cdot, \omega, \zeta) \in C_{i \eta}^{+}$ for $i=1, \ldots, k-1$, which combined with  the fact that $f$ is $C^{k}$, implying that $R_{k-1}(\cdot, \xi, \omega) \in L^{k-1}\left(P_1H^{s}, P_1C_{(k-1) \eta}^{+}\right)$ and is $C^{1}$ in $\zeta$, where $L^{k-1}\left(P_1H^{s}, P_1C_{(k-1) \eta}^{+}\right)$ stands for the usual space of bounded $k-1$ linear forms. Since we have assumed  that $\beta<(k-1) \eta<\alpha$, we can see  that the above integrals are well-defined.
The fact that $t \rightarrow z\left(\theta_{t} \omega\right)$ has a sublinear growth rate is also used in these analysis. The assumptions $\beta<k \eta<\alpha$ and
$$
K L_f\left(\frac{1}{i \eta-\beta}+\frac{1}{\alpha-i \eta}\right)<1 \quad \text { for all } 1 \leqslant i \leqslant k,
$$
together with the fact and the same argument  we used in the case $k=1$, implying that $D_{\zeta}^{k-1} v_t(\cdot, \cdot, \omega): H\rightarrow L^{k}\left(P_1H^{s}, P_1C_{k \eta}^{+}\right)$ is $C^{1}$. This completes the proof.
\end{proof}

By Theorems \ref{thm2.1} and \ref{thm4.1}, we have the following results.
\begin{thm}\label{thm2.1b}
Assume that $M^{s}(\omega)$  is a $C^{k}$ stable  manifold of the stochastic partial differential equation \eqref{2.3} on $\mathcal{L}$ and $z(\theta_\cdot \omega)$ is $C^{k}$ on $\mathcal{L}$, then $\tilde{M}^{s}(\omega)=\left\{\psi(\xi)+z(\theta_\cdot \omega) \mid \psi \in {M}^{s}(\omega)\right\}$ is a $C^{k}$ stable  manifold of the stochastic retarded partial differential equation \eqref{1} on $\mathcal{L}$.
\end{thm}

\section {Existence and smoothness of unstable manifolds}
In the sequel, we are concerned about the existence and smoothness of unstable manifolds for \eqref{1}.
\begin{thm}\label{thm5.1} Assume that $L_f$ is the one defined in $\mathbf{Hypothesis \    A1-A3}$ and the constants $K, \alpha, \beta$ are defined in \eqref{3}. The following statements hold.\\
(i)If $K, \alpha, \beta$ satisfy
$$
K L_f \left(\frac{1}{\eta-\beta}+\frac{1}{\alpha-\eta}\right)<1,
$$
then the delayed random partial differential  equation \eqref{2.3} admits a Lipschitz unstable  invariant  manifold $M^{u}(\omega)=P_1 N^{u}(\omega)$, which can be represented by
$$
M^{u}(\omega)=\left\{\zeta+h^{u}(\zeta, \omega)\mid \zeta\in P_1 P^{u}N^{u}(\omega)\right\}.
$$
Here, $h^{u}(\cdot,\omega): P_1 P^{u}N^{u}(\omega) \rightarrow P_1 P^{s}N^{u}(\omega)$ is a Lipschitz continuous mapping that satisfies $h^{u}(0,\omega)=0$.\\
(ii) If $\beta<k \eta<\alpha$ and
$$
K L_f\left(\frac{1}{i \eta-\beta}+\frac{1}{\alpha-i \eta}\right)<1 \quad \text { for all } 1 \leqslant i \leqslant k,
$$
then the invariant stable manifold $M^{u}(\omega)$ for the delayed random partial differential equation  \eqref{2.3}  is   $C^{k}$, that is, $h^u(\zeta, \omega)$ is $C^{k}$ in $\zeta$.
\end{thm}

\begin{thm}\label{thm5.2}
Assume that $M^{u}(\omega)$  is a Lipschitz unstable  manifold of the stochastic partial differential equation \eqref{2.3} on $\mathcal{L}$, then $\tilde{M}^{u}(\omega)=\left\{\psi(\xi)+z(\theta_\cdot \omega) \mid \psi \in {M}^{s}(\omega)\right\}$ is a Lipschitz unstable  manifold of the stochastic retarded partial differential equation \eqref{1} on $\mathcal{L}$. Moreover, if $M^{u}(\omega)$  is  $C^{k}$ and $z(\theta_\cdot \omega)$ is $C^{k}$ on $\mathcal{L}$, then $\tilde{M}^{u}(\omega)=\left\{\psi(\xi)+z(\theta_\cdot \omega) \mid \psi \in {M}^{u}(\omega)\right\}$ is a $C^{k}$ unstable  manifold of the stochastic retarded partial differential equation \eqref{1} on $\mathcal{L}$.
\end{thm}

The proof of Theorems \ref{thm5.1} and \ref{thm5.2} are quite similar to the results in Section 3 with only a few modifications. In the sequel, we point out the differences and outline the proof procedure while omitting the details, which is the same as Section 3 and Section 4.

Corresponding to space $C_{\eta}^{+}$, we define the Banach space for each $\beta<\eta<\alpha$,
$C_{\eta}^{-}=\left\{\phi:(-\infty, 0] \rightarrow H \mid \phi\right.$ is continuous and $\left.\sup _{t \leqslant 0} e^{-\eta t}\|\phi\|<\infty\right\}$
with the norm
$$
|\phi|_{C_{\eta}^{-}}=\sup _{t \leqslant 0} e^{-\eta t}|\phi_t|.
$$
Denote by
$$
N^{u}(\omega)=\left\{\psi \in H  \mid V\left(t, \omega, \psi\right) \in C_{\eta}^{-}\right\}.
$$

In order to show that $M^{u}(\omega)=\left\{\zeta+h^{u}(\zeta, \omega)\mid \zeta\in P_1P^u N^{u}(\omega)\right\}$ is the Lipschitz stable  invariant unstable manifold of \eqref{2.3}. It follows from the proof of Theorems 5.1 and 5.3 in \cite{DLS04} that   $\psi \in N^{u}(\omega)$ if and only if there exists a function $V(\cdot) \in C_{\eta}^{+}$ with $V(0)=(\phi, \phi(0))\triangleq\psi$ such that
\begin{equation}\label{4.1}
\begin{aligned}
V(t) =&  \tilde{S}(t)\hat{\zeta}+\int_{0}^{t}P^{u} \tilde{S}(t-s)F(s, \theta_s\omega, V(s))d s+\int_{\infty}^{t} P^{s}  \tilde{S}(t-s)  F(s,\theta_s\omega, V(s)) ds, \\
\end{aligned}
\end{equation}
where $\hat{\zeta}=P^u \psi\in H^u$.
By Lemma \ref{lem2.3}, \eqref{4.1} is equivalent to
\begin{equation}\label{4.2}
\begin{aligned}
v_t(\cdot,\omega, \phi)=&P_1V(t,\omega,\psi) \\
=&P_1 \tilde{S}(t)\hat{\zeta}+\int_{0}^{t}P_1P^{u} \tilde{S}(t-s)F(s,\theta_s\omega, P_1^{-1}v_s)d s+\int_{\infty}^{t}P_1P^{s} \tilde{S}(t-s)  F(s,\theta_s\omega, P_1^{-1}v_s) ds. \\
\end{aligned}
\end{equation}
Hence, $\psi \in N^{u}(\omega)$ if and only if there exists a function $v_t \in P_1C_{\eta}^{+}$ with $v_0=\phi$ such that \eqref{4.2} holds.
Taking $t=0$ in \eqref{4.2} we have
\begin{equation}\label{4.3a}
\begin{aligned}
v_0(\cdot,\omega, \phi)=&P_1V(0,\omega,\psi) \\
=&P_1\hat{\zeta}+\int_{\infty}^{0}P_1P^{s}  \tilde{S}(-s)  F(s,\theta_s\omega, P_1^{-1}v_s) ds \\
=&\zeta+h^{u}(\zeta, \omega),
\end{aligned}
\end{equation}
where $\zeta\in P_1P^uN^U(\omega)$ and $h^{u}(\zeta, \omega)$ is defined by
$$
h^{u}(\zeta, \omega)=\int_{\infty}^{0}P_1 P^{s} \tilde{S}(-s)  F(s,\theta_s\omega,P_1^{-1}v_s(\cdot,\omega,\zeta))\mathrm{d} s
$$
satisfying $h^{u}(0, \omega)=0$. Moreover,   $h^{u}$ is measurable. In the same fashion as the case for the smoothness of stable manifold, one may show that $h^{u}$ is $C^{k}$ when the assumptions in Theorem \ref{thm5.2} hold.

\section{Applications}
In this part, we apply the obtained results to the following stochastic single-species age-structured population model
\begin{equation}\label{6.1}
\begin{cases}\partial_t u(x, t)=\Delta u(x, t)-\mu u(x, t)-\delta u(x, t-\tau)+f(u(x, t-\tau))+\sum_{j=1}^{m} g_{j}(x)\frac{\mathrm{d} w_{j}(t)}{\mathrm{d}t}, & x \in[0, \pi], t \geq 0, \\ u(0, t)=u(\pi, t)=0, & t \geq 0, \\ u(x, t)=\phi(x, t), x \in[0, \pi], & t \in[-\tau, 0].\end{cases}
\end{equation}
This equation describes the dynamics of a single-species population distributed over the interval $[0, \pi]$, where $u(x, t)$ represents the population size at $(x, t)$, while $\Delta u(x, t)$ represents the spatial diffusion of the species,  $\mu$ and $\delta$ are positive constants measuring the death rates of mature and  immature individuals respectively, $f$ is the birth function and $\sum_{j=1}^{m} g_{j}(x)\frac{\mathrm{d} w_{j}(t)}{\mathrm{d}t}$ is the noise, respectively.

Let  $z\left(\theta_{t} \omega\right)(x)=\sum_{j=1}^{m} g_{j}(x) z_{j}\left(\theta_{t} \omega_{j}\right)$ with $z_{j}\left(\theta_{t} \omega_{j}\right)$ being defined by \eqref{2.2} and take the transformation $v(x,t)=u(x,t)-z\left(\theta_{t} \omega\right)(x)$. Then,  simple computation gives
\begin{equation}\label{6.2}
\displaystyle \frac{dv(x,t)}{dt}=\displaystyle  \Delta v(x,t)-\mu v(x,t)-\delta v_t(x,t-\tau) -\delta z(\theta_{t-\tau}\omega)(x)+f\left(v(x,t-\tau)+z(\theta_{t-\tau}\omega)\right)+\Delta z(\theta_{t}\omega)(x).
\end{equation}
In the following, we consider the linear part of  \eqref{6.2}, i.e., the following linear random partial functional differential equation
\begin{equation}\label{5.6}
\displaystyle \frac{d\tilde{v}(x,t)}{dt}=\displaystyle  \Delta\tilde{v}(x,t)-\mu \tilde{v}(x,t)-\delta \tilde{v}(x,t-\tau).
\end{equation}
Define $\mathbb{X}:=L^2(0, \pi)$ and $\mathcal{L}\triangleq L^2([-\tau,0], \mathbb{X})$. Then, the operator $B:=\Delta-\mu$ with Dirichlet boundary conditions and $D(B):=H_0^1(0, \pi) \cap H^2(0, \pi)$. Hence, it generates an analytic and compact semigroup $(T(t))_{t \geq 0}$ on $\mathbb{X}$
 The operator $L: W^{1,2}([-1,0], \mathbb{X}) \rightarrow \mathbb{X}$ defined as $L\phi:=-\delta \phi(-\tau)$ and hence it is clear that $\mathbf{Hypothesis \ A1}$ holds. By Theorem 3.29 and Remark 4.1 in \cite{BP}, $\mathbf{Hypothesis \ A2}$ also holds.

Define
$$
\tilde{A}:=\left(\begin{array}{cc}
\frac{d}{d t} & 0 \\
L & B
\end{array}\right),
$$
with domain
$$
\begin{aligned}
D(\tilde{A})=\{(\phi, h) \in H: \phi & \text { is differentialble on }[-\tau, 0],  \dot{\phi} \in \mathcal{L} \text { and } h=\phi(0) \in D(B) \}
\end{aligned}
$$
Since $B$ is compact, by Lemma 4.5  in \cite{BP}, the resolvent $R\left(\lambda, \Delta-\mu I+L_\lambda\right)$ of $\tilde{A}$ is compact for all $\lambda \in \rho\left(\Delta-\mu I+L_\lambda\right)$. Moreover, the spectrum of $\tilde{A}$
has only point spectrum, which is the root of the following characteristic equation
 \begin{equation}\label{5.7}
\displaystyle \lambda+\delta e^{-\lambda \tau}=-n^2 -\mu, \quad n \in \mathbb{N}.
\end{equation}
By Theorem 1.10 in \cite{WJ}, for each given $n$, there exists a real number $\beta$ such that $\operatorname{Re} \lambda \leq \beta$ for all $\lambda \in \sigma (\tilde{A})$. Moreover, if $\alpha$ is a given real number, then there exists only a finite number of $\lambda \in \sigma (\tilde{A})$ such that $\beta \leq \operatorname{Re} \lambda$, implying that $\mathbf{Hypothesis \ A3}$ holds. Therefore, it follows from Lemma \ref{lem3.1} that  the semigroup $\tilde{S}(t)$ satisfies  a  pseudo exponential dichotomy condition, i.e. for the above $\alpha$ and $\beta$ there exist $K$ and a continuous projection $P^{u}$ on $H$ such that
\begin{equation}\label{5.7}
\begin{aligned}
\left\|\tilde{S}(t)P^{u} x\right\| & \leq K e^{\alpha t}\|x\|, & & t \leqslant 0, \\
\left\|\tilde{S}(t)P^{s} x\right\| & \leq K e^{\beta t}\|x\|, & & t \geq 0,
\end{aligned}
\end{equation}
where $P^{s}=I-P^{u}$, $H^{s}=P^{s} H$ and $H^{u}=P^{u}H$. Hence, by the results in Sections 3 and 4, we have the following results.

\begin{thm}\label{thm6.1} Assume that $f$ is globally Lipschitz  and $K, \alpha, \beta$ satisfy
$$
K L_f \left(\frac{1}{\eta-\beta}+\frac{1}{\alpha-\eta}\right)<1.
$$
Then the SPFDE \eqref{6.1} admits both Lipschitz stable and unstable  invariant  manifolds. Moreover,  if    $f$ is $C^k$, $\beta<k \eta<\alpha$ and
$$
K L_f\left(\frac{1}{i \eta-\beta}+\frac{1}{\alpha-i \eta}\right)<1 \quad \text { for all } 1 \leqslant i \leqslant k,
$$
then the invariant stable and unstable manifold  for the SPFDE  \eqref{6.1}  is   $C^{k}$.
\end{thm}

\section{Summary}
In this paper, we have obtained the existence and smoothness of both stable and unstable manifolds for the stochastic retarded partial differential equation \eqref{1} with an additive noise. Our results can be regarded as a first attempt to investigate the invariant structure of stochastic partial functional differential equations of parabolic type. However, there exist some stochastic hyperbolic partial differential equations with delay, such as stochastic delayed wave equations, that are quite difficult to obtain the variation of  constants formula in the natural phase space and needless to say  split the phase space into stable and unstable subspaces. Thus, we will study the invariant manifolds for stochastic wave equations with delay by perturbation methods. Furthermore, the natural phase space for stochastic partial functional differential equations is Banach space and thus establishing the general framework for analyzing existence and smoothness of invariant structures for   RDSs in Banach space is of paramount significance. The noise discussed in this paper is the standard two-sided real-valued Wiener process and we show the equation generates a RDS by a conjugation. Nevertheless, for general noise in infinite dimensional spaces, under what conditions do the perturbed partial functional differential equations generate RDSs and have stable or unstable manifolds remain largely open, which will also be studied in the near future.
\section{Acknowledgement}
This work was jointly supported by the National Natural Science Foundation of China (62173139), China Postdoctoral Science Foundation (2019TQ0089), Hunan Provincial Natural Science Foundation of China (2020JJ5344) the Science and Technology Innovation Program of Hunan Province (2021RC4030), the Scientific Research Fund of Hunan Provincial Education Department (20B353). \\ This work was completed when Wenjie Hu was visiting the Universidad de Sevilla as a visiting scholar, and he would like to thank the staff in the Facultad de Matem\'{a}ticas for
their hospitality  and thank the university for its excellent facilities and support during his stay.\\
The research of T. Caraballo has been partially supported by Spanish Ministerio de Ciencia e
Innovaci\'{o}n (MCI), Agencia Estatal de Investigaci\'{o}n (AEI), Fondo Europeo de
Desarrollo Regional (FEDER) under the project PID2021-122991NB-C21.\\
The authors would also like to thank Professor Weinian Zhang for valuable discussions and suggestions.

\end{document}